\documentclass[11pt]{amsart}

\usepackage{ucs}\usepackage[utf8x]{inputenc}
\usepackage{amsfonts,amssymb,amsthm,amsmath,amsxtra,amscd,verbatim,eucal}
\usepackage{enumerate}
\usepackage[all]{xy}
\usepackage[dvips]{graphics}
\usepackage{psfrag}
\usepackage{hyperref}
\usepackage{framed}
\usepackage{marvosym}

\theoremstyle{plain}
\newtheorem{thm}{Theorem}[section]

\newtheorem*{thm*}{Theorem}
\newtheorem*{prop*}{Proposition}

\newtheorem{lem}[thm]{Lemma}
\newtheorem{prop}[thm]{Proposition}
\newtheorem{cor}[thm]{Corollary}

\theoremstyle{definition}
\newtheorem{defi}[thm]{Definition}
\newtheorem{rmk}[thm]{Remark}
\newtheorem{eg}[thm]{Example}

\newtheorem{question}[thm]{Question}

\DeclareMathOperator{\cd}{cd}

\DeclareMathOperator{\Eff}{Eff}

\DeclareMathOperator{\Amp}{Amp}

\DeclareMathOperator{\red}{red}

\DeclareMathOperator{\ann}{ann}

\DeclareMathOperator{\height}{ht}

\newcommand{\NN}{\ensuremath{\mathbb N}}
\newcommand{\PP}{\ensuremath{\mathbb P}}
\newcommand{\QQ}{\ensuremath{\mathbb Q}}
\newcommand{\RR}{\ensuremath{\mathbb R}}
\newcommand{\ZZ}{\ensuremath{\mathbb Z}}
\newcommand{\CC}{\ensuremath{\mathbb C}}
\newcommand{\OO}{\ensuremath{\mathcal O}}

\newcommand\lra{\longrightarrow}

\newcommand{\la}{\ensuremath{\lambda}}

\newcommand{\shf}{\ensuremath{{\mathcal F}}}
\newcommand{\she}{\ensuremath{{\mathcal E}}}
\newcommand{\sha}{\ensuremath{{\mathcal A}}}

\newcommand{\shn}{\ensuremath{{\mathcal N}}}

\newcommand{\shl}{\ensuremath{{\mathcal L}}}

\newcommand{\shk}{\ensuremath{{\mathcal K}}}

\newcommand{\shr}{\ensuremath{{\mathcal R}}}

\newcommand{\hh}[3]{\ensuremath{h^{#1}\left(#2,#3\right)}}

\newcommand{\HH}[3]{\ensuremath{H^{#1}\left(#2,#3\right)}}
\newcommand{\HHv}[3]{\ensuremath{H^{#1}(#2,#3)}}
\newcommand{\ha}[3]{\ensuremath{\widehat{h}^{#1}\left(#2,#3\right)}}

\newcommand{\ses}[3]{\ensuremath{0\rightarrow #1 \rightarrow #2 \rightarrow #3 \rightarrow 0}}

\newcommand{\zj}[1]{\ensuremath{ \left( #1 \right) }}

\newcommand{\st}[1]{\ensuremath{ \left\{ #1 \right\} }}

\newcommand{\deq}{\ensuremath{ \stackrel{\textrm{def}}{=}}}

\newcommand{\shll}[1]{\ensuremath{\shl^{\otimes #1}}}

\def\.{\cdot}
\def\^{\widehat}
\def\~{\widetilde}

\newcommand\newop[2]{\def#1{\mathop{\rm #2}\nolimits}}
\newop\voll{vol}
\newop\NS{NS}
\newop\Neg{Neg}
\newop\Null{Null}
\newop\Pic{Pic}
\newop\Bstab{B_{+}}
\newop\Bst{B_{stab}}
\newop\Bres{B_{restr}}
\newop\Bplus{\mathbf{B}_+}
\newop\Bminus{\mathbf{B}_-}
\newop\Exc{Exc}
\newop\B{\mathbf{B}{}}
\newop\Bs{Bs}
\newop\End{End}
\newop\Amp{Amp}
\newop\Face{Face}
\newop\BigCone{Big}
\newop\index{ind}
\newop\reg{reg}

\newcommand{\equ}{\ensuremath{\,=\,}}

\newcommand{\dleq}{\ensuremath{\,\leq\,}}
\newcommand{\dgeq}{\ensuremath{\,\geq\,}}

\newcommand{\dsubseteq}{\ensuremath{\,\subseteq\,}}

\newop\Gal{Gal}

\numberwithin{equation}{section} 

\begin{document}

\title{Partial positivity: geometry and cohomology of $\mathbf{q}$-ample line bundles}

\author{Daniel Greb}
\address{Daniel Greb, Ruhr-Universit\"at Bochum, Fakult\"at f\"ur Mathematik, Arbeitsgruppe Al\-ge\-bra/Top\-o\-lo\-gie, 44780 Bochum, Germany}
\email{{\tt daniel.greb@ruhr-uni-bochum.de}}

\author{Alex K\"uronya}
\address{Alex K\"uronya, Budapest University of Technology and Economics, Department of Algebra, P.O. Box 91, H-1521 Budapest, Hungary}
\address{Albert-Ludwigs-Universit\"at Freiburg, Mathematisches Institut, Eckerstra\ss e 1, D-79104 Frei\-burg, Germany}
\email{{\tt alex.kueronya@math.uni-freiburg.de}}

\begin{abstract}
We give an overview of partial positivity conditions for line bundles, mostly from a cohomological point of view. Although the current work is to a large extent of expository nature, 
we present some minor improvements over the existing literature and a new result: a Kodaira-type vanishing theorem for effective $q$-ample Du Bois divisors and log canonical pairs.
\end{abstract}

\dedicatory{To Rob Lazarsfeld on the occasion of his 60\textsuperscript{th} birthday}

\maketitle

\tableofcontents

\section{Introduction}

Ampleness is one of the central notions of  algebraic geometry, possessing  the  extremely useful feature  that it has geometric, numerical, and 
cohomological characterizations. Here we will concentrate on its cohomological side. The fundamental result in this direction
is the theorem of Cartan--Serre--Grothendieck (see \cite[Theorem 1.2.6]{PAG}): for  a complete projective scheme $X$, and a line bundle  $\shl$ 
on $X$, the following are equivalent to $\shl$ being ample:
\begin{enumerate}
 \item There exists a positive integer $m_0=m_0(X,\shl)$ such that $\shl^{\otimes m}$ is very ample for all $m\geq m_0$. 
\item For every coherent sheaf $\shf$ on $X$, there exists a positive integer $m_1=m_1(X,\shf,\shl)$ for which $\shf\otimes\shl^{\otimes m}$
is globally generated for all $m\geq m_1$.
\item For every coherent sheaf $\shf$ on $X$, there exists a positive integer $m_2=m_2(X,\shf,\shl)$ such that 
\[
 \HH{i}{X}{\shf\otimes\shl^{\otimes m}} \equ \{0\}
\]
for all $i\geq 1$ and all $m\geq m_2$.
\end{enumerate}

We will focus on the direction pointed out by Serre's vanishing theorem, part (3) above, and concentrate on line bundles with vanishing cohomology above a certain degree. 

Historically, the first result in this direction is due to Andreotti and Grauert \cite{AG}. They prove 
that given a compact complex manifold $X$ of dimension $n$, and a holomorphic line bundle $\shl$ on $X$ equipped with a Hermitian metric whose curvature 
is a $(1,1)$-form with at least $n-q$ positive eigenvalues at every point of $X$, then for every coherent sheaf $\shf$ on $X$, there exists a natural number $m_0(\shl,\shf)$ such that 
\begin{equation}\label{eq:definingqvanishing}
 \HH{i}{X}{\shf\otimes\shl^{\otimes m}} \equ \{0\}\quad \quad \text{ for all } m \geq m_0(\shl,\shf) \text{ and for all }i>q. 
\end{equation}
In \cite{DPS}, Demailly, Peternell, and Schneider posed the question under what circumstances the converse does hold. I.e., they asked: assume that for every coherent sheaf $\shf$ there exists $m = m_0(\shl,\shf)$ such that the vanishing \eqref{eq:definingqvanishing} holds. Does $\shl$ admit a Hermitian metric with the expected number of positive eigenvalues? In dimension two, Demailly \cite{Dem_AG} proved an asymptotic version of this converse to the Andreotti--Grauert Theorem using tools related to asymptotic cohomology; subsequently, Matsumura \cite{Mats11} gave a positive answer to the question for surfaces. However, there exist higher-dimensional counterexamples to the converse Andreotti--Grauert problem in the range $\tfrac{\dim X}{2}-1<q<\dim X-2$, constructed by Ottem \cite{Ottem}.

We will study line bundles with the property of the conclusion of the Andreotti--Grauert theorem; let $X$ be a complete scheme of dimension $n$, 
$0\leq q\leq n$ an integer. A line bundle $\shl$ is called \emph{naively  $q$-ample}, 
or simply \emph{$q$-ample} if for every coherent sheaf $\shf$ on $X$ there  exists an integer $m_0=m_0(\shl,\shf)$ for which
 \[
 \HH{i}{X}{\shf\otimes\shl^{\otimes m}} \equ \{0\}
\]
for $m\geq m_0$ and for all $i>q$. 

It is a consequence of the Cartan--Grothendieck--Serre result discussed above that $0$-ample\-ness reduces to the usual notion of ampleness. 
In \cite{DPS}, the authors studied naive $q$-ampleness along with various other notions of partial cohomological positivity. 
Part of their approach is to look at positivity of restrictions to elements of a complete flag, and they use it to give 
a partial vanishing theorem  similar to that of Andreotti--Grauert. 

In a beautiful paper, \cite{Totaro}, Totaro  proves that the competing partial positivity concepts are in fact all equivalent in characteristic zero, 
thus laying down the foundations for  a very satisfactory theory. His result goes as follows: let $X$ be a projective scheme of dimension $n$, 
$\sha$ a sufficiently ample line bundle on $X$, $0\leq q\leq n$ a natural number.  Then for all line bundles $\shl$ on $X$,
 the following are equivalent:
\begin{enumerate}
\item $\shl$ is naively $q$-ample
\item $\shl$ is uniformly $q$-ample, that is, there exists a constant $\la>0$ such that 
\[
 \HH{i}{X}{\shl^{\otimes m}\otimes\sha^{\otimes -j}} \equ \{0\}
\]
for all $i>q$, $j>0$, and $m\geq \la j$.
 \item There exists a positive integer $m_1=m_1(\shl,\sha)$ such that 
\[
 \HH{q+i}{X}{\shl^{\otimes m_1}\otimes\sha^{\otimes -(n+i)}} \equ \{0\}
\]
for all $1\leq i\leq n-q$.
\end{enumerate}

As an outcome, Totaro can prove that on the one hand, for a given $q$, the set of $q$-ample line bundles forms a open cone in the N\'eron--Severi space, and on the other hand that $q$-ampleness is an open property in families as well.

In his recent article \cite{Ottem}, Ottem works out other basic properties of $q$-ample divisors and employs it to study subvarieties of higher codimension. We will give an overview of his results in Section~\ref{subsect:Lefschetz}.

Interestingly enough, prior to \cite{DPS}, Sommese \cite{Sommese} defined a geometric version of partial ampleness by studying the dimensions of the fibres of the morphism associated to a given
line bundle. In the case of semi-ample line bundles, where Sommese's notion is defined, he proves his condition to be equivalent to 
naive $q$-ampleness.  Although more limited in scope, Sommese's geometric notion extends naturally to vector bundles as well. 

One of the major technical vanishing theorems for ample divisors that does not follow directly from the definition is Kodaira's vanishing theorem: if $X$ is a smooth projective variety, defined over an algebraically closed field of characteristic zero, and $\shl$ is an ample line bundle on $X$, then 
\[\HH{j}{X}{\omega_X \otimes \shl} = \{0\} \quad \text{ for all }j > 0.\] We refer the reader to \cite{Kodaira} for the original analytic proof, to \cite{DeligneIllusie} for a subsequent algebraic proof, to \cite{Raynaud} for a counterexample in characteristic $p$, and to \cite[Chapter 9]{KollarShafarevich} as well as to \cite{EsnaultViehweg} for a general discussion of vanishing theorems. It is a natural question to ask whether an analogous vanishing holds for $q$-ample divisors. While one of the ingredients of classical proofs of Kodaira vanishing, namely Lefschetz' hyperplane section theorem, has been generalized to the $q$-ample setup by Ottem \cite[Cor.~5.2]{Ottem}, at the same time he gives a counterexample to the Kodaira vanishing theorem for $q$-ample divisors, which we recall in Section~\ref{subsect:Lefschetz}. In Ottem's example, the chosen $q$-ample divisor is not pseudo-effective. In Section~\ref{sect:qKodaira} we show that there is a good reason for this: we prove that $q$-Kodaira vanishing 
holds for reduced effective divisors which are not to singular; more precisely, we prove two versions of $q$-Kodaira vanishing which are related to log canonicity and the Du Bois condition.

\begin{thm*}[= Theorem~\ref{thm:reduceddivisor}]
 Let $X$ be a normal proper variety, $D\subset X$ a reduced effective (Cartier) divisor such that
\begin{enumerate}
 \item $\mathcal{O}_X(D)$ is $q$-ample,
 \item $X \setminus D$ is smooth, and $D$ (with its reduced subscheme structure) is Du Bois.
\end{enumerate}
Then, we have
\begin{align*}
 &\HH{j}{X}{\omega_X \otimes \mathcal{O}_X(D)}  = \{0\} \quad \text{ for all } j> q,
\intertext{as well as}
&\HH{j}{X}{\mathcal{O}_X(-D)} = \{0\} \quad \text{ for all }j < \dim X - q .
\end{align*}
\end{thm*}
\begin{thm*}[= Theorem~\ref{thm:QQversion}]
  Let $X$ be a proper Cohen-Macaulay variety, $\shl$ a $q$-ample line bundle on $X$, and $D_i$ different irreducible Weil divisors on $X$. Assume that $\shl^{m} \cong \mathcal{O}_X(\sum d_jD_j)$ for some integers $1 \leq d_j < m$. Set $m_j := m/\mathrm{gcd}(m, d_j)$. Assume furthermore that the pair $\bigl(X, \textstyle\sum (1- \frac{1}{m_j})D_j \bigr)$ is log canonical.
Then, we have
\begin{align*}
 &\HH{j}{X}{\omega_X \otimes \shl}  = \{0\} \quad \text{ for all } j> q,
\intertext{as well as}
&\HH{j}{X}{\shl^{-1}} = \{0\} \quad \text{ for all }j < \dim X - q .
\end{align*}
\end{thm*}
\subsubsection*{Global conventions}
If not mentioned otherwise, we work over the complex numbers, and all divisors are assumed to be Cartier. 
\subsubsection*{Acknowledgments} Both authors were partially supported by the DFG-Forschergruppe 790 ''Classification of Algebraic Surfaces and Compact Complex Manifolds``, as well as by the DFG-Graduiertenkolleg 1821 ''Cohomological Methods in Geometry``. The first named author gratefully acknowledges additional 
support by the Baden-W\"urttemberg-Stiftung through the ''Eliteprogramm f\"ur Postdoktorandinnen und Postdoktoranden``. The second-named author
was in addition partially supported by the OTKA grants 77476 and  81203 of the Hungarian Academy of Sciences. 

\section{Overview of the theory of $q$-ample line bundles}

\subsection{Vanishing of cohomology groups and partial ampleness}
Starting with the pioneering work \cite{DPS} of Demailly--Peter\-nell--Schneider related to the Andreotti--Grauert problem, there has been a certain interest in studying line bundles 
with vanishing cohomology above a given degree. Just as big line bundles are a generalization of ample ones along its geometric side, these so-called $q$-ample bundles focus on a weakening 
of the cohomological characterization of ampleness. In general, there exist competing definitions of various flavours, 
which were shown to be equivalent in characteristic zero  in \cite{Totaro}.

\begin{defi}[Definitions of partial ampleness]\label{defi:3 partial ampleness}
Let $X$ be a complete scheme of dimension $n$ over an algebraically closed field of arbitrary characteristic, $\shl$ an invertible sheaf on $X$, $q$ a natural number.
\begin{enumerate}
 \item The invertible sheaf $\shl$ is called \emph{naively $q$-ample}, if for every coherent sheaf $\shf$ on $X$ there exists a natural number $m_0=m_0(\shl,\shf)$ having the 
property that 
\[
 \HH{i}{X}{\shf\otimes\shll{m}} \equ \{0\} \ \ \ \text{for all $i>q$ and $m\dgeq m_0$.}
\]
\item Fix a very ample invertible sheaf $\sha$ on $X$. We call $\shl$ \emph{uniformly $q$-ample} if  there exists a constant $\la=\la(\sha,\shl)$ such that 
\[
 \HH{i}{X}{\shll{m}\otimes\sha^{\otimes -j}} \equ \{0\} \ \ \ \text{for all $i>q$, $j>0$, and $m\dgeq \la\cdot j$.}
\]
\item Fix a Koszul-ample invertible sheaf $\sha$ on $X$. We say that $\shl$ is \emph{$q$-$T$-ample}, if there exists a positive integer $m_1=m_1(\sha,\shl)$ satisfying 

\begin{eqnarray*}
  \HH{q+1}{X}{\shll{m_1}\otimes\sha^{\otimes -(n+1)}} & \equ &  \HH{q+2}{X}{\shll{m_1}\otimes\sha^{\otimes -(n+2)}} \\
  \ldots & = & \HH{n}{X}{\shll{m_1}\otimes\sha^{\otimes -2n+q}} \equ \{0\}
\end{eqnarray*}
\end{enumerate}
An integral Cartier divisor is called \emph{naively $q$-ample/uniformly $q$-ample/$q$-$T$-ample}, if the invertible sheaf $\OO_X(D)$ has the appropriate 
property. 
\end{defi}

\begin{rmk}[Koszul-ampleness]
We recall that  a connected locally finite\footnote{In this context locally finite means that  $\dim R_i<\infty$ for all graded pieces.}
graded ring $R_\bullet = \oplus_{i=0}^{\infty}R_i$ is called $N$-Koszul
for a positive integer  $N$, if the field $k=R_0$ has a resolution 
\[
 \ldots \lra M_1 \lra M_0 \lra k \lra 0
\]
as a graded $R_\bullet$-module, where for all $i\geq N$ the module $M_i$ is free and generated in degree $i$. 

In turn, a very ample line bundle $\sha$ on a projective scheme $X$ (taken to be connected and reduced to arrange that the ring of its regular functions $k \deq \OO_X(X)$ is a field)
is called $N$-Koszul if the section ring $R(X,\sha)$ is $N$-Koszul. The line bundle $\sha$ is said to be \emph{Koszul-ample} if it is $N$-Koszul with $N=2\dim X$. 
It is important to point out that Castelnuovo--Mumford regularity with respect to a  Koszul-ample line bundle has favourable  properties. 

If $\sha$ is an arbitrary ample line bundle on $X$, then Backelin \cite{Backelin} showed that  there exists a positive integer  $k_0\in\NN$ having the property that 
$\sha^{\otimes k}$ is Koszul-ample for all $k\geq k_0$. 
\end{rmk}

\begin{rmk}
Naive $q$-ampleness is the immediate extension of the Grothendieck--Cartan--Serre vanishing criterion for ampleness. Uniform $q$-ampleness first appeared in \cite{DPS}; the term
$q$-$T$-ampleness was coined by Totaro in \cite[Section 7]{Totaro} extending the idea of Castelnuovo--Mumford regularity. 

A line bundle is ample if and only if it is naively $0$-ample, while all line bundles are $n=\dim X$-ample.
\end{rmk}

\begin{eg}\label{ex:blowup}
One source of examples of $q$-ample divisors  comes from ample vector bundles: according to \cite[Proposition 4.5]{Ottem}, if $\she$ is an ample vector bundle of rank $r\leq \dim X$
on a scheme $X$, $s\in\HH{0}{X}{\she}$, then $Y\deq Z(s)\subseteq X$ is an ample subvariety. By definition, this means that $\OO_{X'}(E)$ is $(r-1)$-ample, where 
$\pi:X'\to X$ is the blow-up of $X$ along $Y$ with exceptional divisor $E$, cf.~Definition~\ref{defi:amplesubscheme}.
\end{eg}

\begin{rmk}
A straightforward sufficient condition for (naive) $q$-ampleness can be obtained by studying restrictions of $\shl$ to general complete intersection subvarieties. 
More specifically, the following claim is shown  in \cite[Theorem A]{qAmp}: let $X$ be a projective variety over the complex numbers, $L$ a Cartier divisor, 
$A_1,\dots,A_q$   very ample  Cartier divisors on $X$ such that $L|_{E_1\cap\dots\cap E_q}$ is ample for general $E_j\in |A_j|$. 
Then, for any coherent sheaf $\shf$ on $X$ there exists an integer $m(L,A_1,\dots,A_q,\shf)$ such that 
\[
 \HHv{i}{X}{\shf\otimes \OO_X(mL+N+\sum_{j=1}^{q}k_jA_j)} \equ \{0\}
\]
for all $i>q$, $m\geq m(L,A_1,\dots,A_q,\shf)$,  $k_j\geq 0$, and all nef divisors $N$. In particular, the conditions of the above claim ensure that $L$ is $q$-ample.
\end{rmk}

\begin{rmk}
There are various implications among the three definitions over an arbitrary field. 
As it was verified  in \cite[Proposition 1.2]{DPS} by an argument via resolving coherent sheaves by direct sums  of ample line 
bundles\footnote{Note that the resulting resolution is not guaranteed to be finite; see \cite[Example 1.2.21]{PAG} for a discussion of this possibility.}, 
 a uniformly $q$-ample line bundle is necessarily naively $q$-ample. On the other hand, naive $q$-ampleness implies $q$-$T$-ampleness by definition. 
\end{rmk}

\begin{rmk}
The idea behind the definition of $q$-$T$-ampleness is to reduce the question of $q$-ampleness to the vanishing of finitely 
many cohomology groups. It is not known whether the three definitions coincide in positive characteristic. 
\end{rmk}
The following result of Totaro compares the three different approaches to $q$-ampleness in characteristic zero.
\begin{thm}[Totaro; \cite{Totaro}, Theorem~8.1]\label{thm:Totaro equivalence}
Over a field of characteristic zero, the three definitions of partial ampleness are equivalent. 
\end{thm}

The proof uses on the one hand methods from positive characteristics to generalize a vanishing result of Arapura \cite[Theorem 5.4]{Arapura}, 
at the same time it relies on earlier work  of Orlov and  Kawamata on resolutions of the diagonal via Koszul-ample line bundles.  

\begin{rmk}[Resolution of the diagonal]
One of the main building blocks of \cite{Totaro} is an explicit resolution of the diagonal as a sheaf on $X\times X$ depending on an ample line bundle $\sha$ on
$X$.  This is used as a tool to  prove an important result   on the regularity of tensor products of sheaves (see Theorem~\ref{thm:Totaro CM-regularity} below, itself an 
improvement over a statement of Arapura \cite[Corollary 1.12]{Arapura}), which in turn is instrumental in showing that $q$-T-ampleness implies uniform $q$-ampleness,
the non-trivial part of Totaro's theorem on the equivalence of the various definitions of partial vanishing.

The resolution in question --- which exists over an arbitrary field --- had first been constructed by Orlov \cite[Proposition A.1]{Orlov} under the assumption that $\sha$ is sufficiently ample, and subsequently
improved by Kawamata \cite{Kaw} by making the more precise assumption that the coordinate ring $R(X,\sha)$ is a Koszul algebra. 

Totaro reproves Kawamata's result under the weaker hypothesis that $\sha$ that the section ring $R(X,\sha)$ is $N$-Koszul for some positive integer $N$ (for the 
most of \cite{Totaro} one will set $N = 2\dim X$).

To construct the Kawamata--Orlov--Totaro resolution of the diagonal, we will proceed as follows. First, we define a sequence of $k$-vector spaces $B_i$ by setting
\[
 B_i \deq \begin{cases}
           k & \text{ if } i=0, \\
           \HH{0}{X}{\sha} & \text{ if } i=1, \\
           \ker \zj{B_{i-1}\otimes\HH{0}{X}{\sha} \lra B_{i-2}\otimes \HH{0}{X}{\sha^2} } & \text{ if } i\geq 2 \ .
          \end{cases}
\]
Note that $\sha$ is $N$-Koszul precisely if the following sequence of graded $R(X,\sha)$-modules cooked up from the $B_i$'s is exact: 
\[
 B_N\otimes R(X,\sha)(-N) \lra \ldots \lra B_1\otimes R(X,\sha)(-1) \lra R(X,\sha) \lra k \lra 0\ .
\]
Next, set 
\[
 \shr_i \deq \begin{cases} 
              \OO_X & \text{ if } i=0\ , \\
              \ker \zj{B_i\otimes \OO_X\lra B_{i-1}\otimes \sha} & \text{ if } i>0\ .
             \end{cases}
\]
Totaro's claim \cite[Theorem 2.1]{Totaro} goes as follows: if $\sha$ is an  $N$-Koszul line bundle on $X$, then there exists an exact sequence 
\[
 \shr_{N-1}\boxtimes \sha^{-N+1} \lra \cdots\lra \shr_1\boxtimes \sha^{-1} \lra \shr_0\boxtimes\OO_X \lra \OO_\Delta\lra 0\ ,
\]
where $\boxtimes$ denotes the external tensor product on $X\times X$, and $\Delta\subset X\times X$ stands for the diagonal. 
\end{rmk}

The above construction  leads to a statement of independent interest about the regularity of tensor products of sheaves, which had already appeared in some form
in \cite{Arapura}. 

\begin{thm}[Totaro; \cite{Totaro}, Theorem~3.4]\label{thm:Totaro CM-regularity}
Let $X$ be a connected and reduced projective scheme of dimension $n$, $\sha$ a $2n$-Koszul line bundle, $\she$ a vector bundle, $\shf$ a coherent sheaf on $X$. Then,
\[
 \reg (\she\otimes\shf) \dleq \reg(\she) + \reg(\shf)\ .
\]
\end{thm}

\begin{rmk}
 In the case of $X=\PP^n_\CC$ the above theorem is a simple application of Koszul complexes (see \cite[Proposition 1.8.9]{PAG} for instance). 
\end{rmk}

\begin{rmk}[Positive characteristic methods]
 Another crucial point in the proof of the equivalence of the various definitions of $q$-ampleness is a vanishing result in positive characteristic originating 
 in the work of Arapura in the smooth case \cite[Theorem 5.3]{Arapura} that was extended to possibly singular schemes by Totaro \cite[Theorem 5.1]{Totaro}
 exploiting a flatness property of the Frobenius over arbitrary schemes over fields of prime cardinality. 
 
 The statement is essentially as follows: let $X$ be a connected and reduced projective scheme of dimension $n$ over a field of positive characteristic $p$, 
 $\sha$ a Koszul-ample line bundle on $X$, $q$ a natural number. Let $\shl$ be a line bundle on $X$ satisfying 
 \[
  \HH{q+1}{X}{\shl\otimes\sha^{\otimes (-n-1)}} \equ \HH{q+2}{X}{\shl\otimes\sha^{\otimes (-n-2)}} \equ \ldots \equ \{0\}\ .
 \]
Then, for any coherent sheaf $\shf$ on $X$ one has 
\[
 \HH{i}{X}{\shl^{\otimes p^m}\otimes\shf} \equ \{0\} \ \ \text{for all $i>q$ and $p^m\geq \reg_\sha(\shf)$.}
\]
\end{rmk}

\subsection{Basic properties of $q$-ampleness}
From now on we return to our blanket assumption and work over the complex number field; we will immediately see that the  
equivalence of the various definitions brings all sorts of perks.

First, we point out that $q$-ampleness enjoys many formal properties analogous to ampleness. The following statements have been part of the folklore, for precise proofs we refer the reader to 
\cite[Proposition 2.3]{Ottem} and \cite[Lemma 1.5]{DPS}.

\begin{lem}\label{lem: ample formal}
Let $X$ be a projective scheme, $\shl$ a line bundle on $X$, $q$ a natural number. Then, the following holds.
\begin{enumerate}
 \item $\shl$ is $q$-ample if and only if $\shl|_{X_{\red}}$ is $q$-ample on $X_{\red}$. 
\item $\shl$ is $q$-ample precisely if $\shl|_{X_i}$ is $q$-ample on  $X_i$ for every irreducible component $X_i$ of $X$. 
\item For a finite morphism $f:Y\to X$, if $\shl$ on $X$ is $q$-ample then so is $f^*\shl$. Conversely, if $f$ is surjective as well, then the $q$-ampleness of $f^*\shl$ implies 
the $q$-ampleness of $\shl$. 
\end{enumerate}
\end{lem}

The respective proofs of the ample case go through with minimal modifications. Another feature surviving in an unchanged form  is the fact that to check (naive) $q$-ampleness we can restrict 
our attention to line bundles. 

\begin{lem}
Let $X$ be a projective scheme, $\shl$ a line bundle,  $\sha$ an arbitrary ample line bundle on $X$. Then, $\shl$ is $q$-ample precisely if there exists a natural number $m_0=m_0(\sha,\shl)$ 
having the property that 
\[
 \HH{i}{X}{\shll{m}\otimes\sha^{\otimes -k}} \equ \{0\}
\]
 for all $i> q$, $k\geq 0$, and $m\geq m_0k$. 
\end{lem}
\begin{proof}
Follows immediately by decreasing induction on $q$ from the fact that every coherent sheaf $\shf$ on $X$ has a possibly infinite resolution by finite direct sums of non-positive powers
 of the ample line bundle $\sha$ (\cite[Example 1.2.21]{PAG}).
\end{proof}

Ample line bundles are good to work with for many reasons, but the fact that they are open both in families and in the N\'eron--Severi space contributes considerably. As it turns out, 
the same properties are valid for $q$-ample line bundles as well. 

\begin{thm}[Totaro; \cite{Totaro}, Theorem 9.1]\label{thm:open in families}
Let $\pi:X\to B$ be a flat projective morphism of schemes (over $\ZZ$) with connected fibres, $\shl$ a line bundle on $X$, $q$ a natural number. Then, the subset of points $b$  of $B$
having the property that $\shl|_{X_b}$ is $q$-ample is Zariski open.
\end{thm}
\begin{proof}[Sketch of proof] 
This is one point where $q$-$T$-ampleness plays a role, since in that formulation one only needs to check vanishing for a finite number of cohomology groups.

Assume that $\shl|_{X_b}$ is $q$-$T$-ample for a given point $b\in B$; let $U$ be an affine open neighbourhood on $b\in B$, and $\sha$ a line bundle on $\pi^{-1}(U)\subseteq X$, 
whose restriction to $X_b$  is Koszul-ample. Since Koszul-ampleness is a Zariski-open property, $\sha|_{X_{b'}}$ is again Koszul-ample for an open subset of points $b'\in U$; without loss
 of generality we can assume that this holds on the whole of $U$.

We will use the line bundle $\sha|_{X_{b'}}$ to check $q$-$T$-ampleness of $\shl_{X_{b'}}$ in an open neighbourhood on $b\in U$. By the $q$-$T$-ampleness of $\shl|_{X_b}$ there exists a positive integer $m_0$
satisfying
\[
 \HH{q+1}{X_b}{\shll{m_0}\otimes\sha^{\otimes -n-1}} \equ \ldots \equ \HH{n}{X_b}{\shll{m_0}\otimes\sha^{\otimes -2n+q}} \equ \{0\}\ .
\]
It follows from  the semicontinuity theorem  that the same vanishing holds true for points in an open neighbourhood of $b$. 
\end{proof}

In a different direction, Demailly--Peternell--Schneider proved that uniform $q$-ampleness is open in the N\'eron--Severi space. To make this precise we need the fact that uniform $q$-ampleness
is a numerical property; once this is behind us, we can define $q$-ampleness for numerical equivalence classes of $\RR$-divisors. 

\begin{rmk}
Note that a line bundle $\shl$ is $q$-ample if and only if $\shll{m}$ is $q$-ample for some positive integer $m$. Therefore it makes sense to talk about $q$-ampleness of $\QQ$-Cartier divisors; a $\QQ$-divisor $D$ is said to be $q$-ample, if it has a multiple $mD$ that is integral and 
$\OO_X(mD)$ is $q$-ample. 
\end{rmk}

\begin{thm}[$q$-ampleness is a numerical property]\label{thm:q-ample numerical}
Let $D$ and $D'$ be numerically equivalent integral Cartier divisors on an irreducible complex projective variety $X$, $q$ a natural number. Then,
\[
\text{$D$ is  $q$-ample}\ \ \Leftrightarrow \ \ \text{$D'$ is  $q$-ample}.
\]
\end{thm}

Demailly--Peternell--Schneider in \cite[Proposition 1.4]{DPS} only prove this claim for smooth projective varieties. 
 The proof in \cite{DPS} cites the completeness of $\Pic^0(X)$,
 hence it is far from obvious how to extend it. Here we give a proof that is valid under the given more general  hypothesis. Instead of dealing with uniform $q$-ampleness we use the naive 
formulation. 
\begin{proof}
Let $\shn'$ be a numerically trivial line bundle, $\shl$ a $q$-ample line bundle on $X$. This means that for a given coherent sheaf $\shf$, we have 
\[
 \HH{i}{X}{\shll{m}\otimes\shf} \equ \{0\} \ \ \ \text{for $i>q$ and $m_0=m_0(\shl,\shf)$.}
\]
We need to prove that 
\[
 \HH{i}{X}{(\shl\otimes\shn')^{\otimes m}\otimes\shf} \equ \{0\} 
\]
holds for all $i>q$, and for suitable $m\geq m_1(\shl,\shn',\shf)$. To this end, we will study the function
\[
f_i^{(m)}\ :\  \shn\  \mapsto\  \hh{i}{X}{\shll{m}\otimes\shn\otimes\shf}
\]
as a function on the closed points on the subscheme $\mathcal{X}$ of the Picard scheme that parametrizes numerically trivial line bundles on $X$,
which is a scheme of finite type by the boundedness of numerically trivial line bundles  \cite[Theorem 1.4.37]{PAG}. We know that 
\[
 f_i^{(m)}(\OO_X) \equ \hh{i}{X}{\shll{m}\otimes\shf} \equ \{0\}
\]
for $i>q$, and $m\geq m_0$. By the semicontinuity theorem, $f_i^{(m)}$ attains the same value on a dense open subset of $\mathcal X$. 

By applying Noetherian induction and semicontinuity on the irreducible 
components of the complement (on each of which we apply the $q$-ampleness of $\shl$ for coherent sheaves of the shape $\shf\otimes\shn$, 
$\shn$ numerically trivial) we will eventually find a value $m_0'=m_0'(\shl,\shf)$ such that 
\[
 f_i^{(m)} \ \equiv\  0
\]
for all $i>q$ and $m\geq m_0'$.

 But this implies that 
\[
  \HH{i}{X}{(\shl\otimes\shn')^{\otimes m}\otimes\shf} \equ \HH{i}{X}{\shll{m}\otimes ((\shn')^{\otimes m})\otimes\shf} \equ \{0\} 
\]
for $m\geq m_0'$, since the required vanishing holds for an arbitrary numerically trivial divisor in place of $(\shn')^{\otimes m}$. 
\end{proof}

\begin{rmk}
As a result, we are in a position to  extend the definition of $q$-ampleness elements of $N^1(X)_\QQ$: if $\alpha$ is a numerical equivalence class of $\QQ$-divisors, then we will call it
$q$-ample, if one (equivalently: all) of its representatives are $q$-ample. 
\end{rmk}

\begin{rmk}
 Since Fujita's vanishing theorem holds over algebraically closed fields  of  positive characteristic as well 
(see \cite{Fujita} or \cite[Remark 1.4.36]{PAG}), the boundedness of numerically trivial line bundles holds again in this case by
\cite[Proposition 1.4.37]{PAG} and  we can conclude that naive $q$-ampleness is invariant with respect to numerical equivalence in that situation as well. 
\end{rmk}

\begin{defi}[$q$-ampleness for $\RR$-divisors]
An $\RR$-divisor $D$ on a complex projective variety is \emph{$q$-ample}, if 
\[
 D \equ D' + A\ ,
\]
 where $D'$ is a $q$-ample $\QQ$-divisor, $A$ an ample $\RR$-divisor. 
\end{defi}

The result that $q$-ample $\RR$-divisors form an open cone in $N^1(X)_\RR$ was proved in \cite{DPS}. Here we face the same issue as with Theorem~\ref{thm:q-ample numerical}: in the article
\cite{DPS} only the smooth case is considered, and the proof given there does not seem to generalize to general varieties. Here we present a proof of the general case. 

\begin{defi}
Given $\alpha\in N^1(X)_\RR$, we set 
\[
 q(\alpha) \deq \min\st{q\in\NN\,|\, \text{$\alpha$ is $q$-ample}}\ .
\]
\end{defi}

\begin{thm}[Upper-semicontinuity of $q$-ampleness]\label{thm:q-ample open} 
Let $X$ be an irreducible projective variety over the complex numbers. Then, the function
\[
 q\ :\ N^1(X)_\RR \lra \NN
\]
is upper-semicontinuous. In particular, for a given $q\in\NN$, the set of $q$-ample classes forms an open cone. 
\end{thm}

In order to be able to prove this result, we need some auxiliary statements. To this end, Demailly--Peternell--Schneider introduce the concept of height of coherent sheaves with respect to
a given ample divisor. Roughly speaking the height of a coherent sheaf tells us, what multiples of the given ample divisor we need to obtain a linear resolution.

\begin{defi}[Height]\label{defi:height}
Let $X$ be an irreducible projective variety, $\shf$ a coherent sheaf, $\sha$ an ample line bundle. Consider the set $\shr$ of all resolutions  
\[
 \ldots\to\bigoplus_{1\leq l\leq m_k}\sha^{\otimes - d_{k,l}} \to\ldots\to \bigoplus_{1\leq l\leq m_0}\sha^{\otimes -d_{0,l}} \to\shf\to 0\ 
\]
of $\shf$ by non-positive powers of $\sha$ (that is, $d_{k,l}\geq 0$).
Then,
\[
 \height_\sha(\shf) \deq \min_{\shr} \max_{0\leq k\leq \dim X,1\leq l\leq m_k} d_{k,l}\ .
\]
\end{defi}

\begin{rmk}
On could define the height by looking at resolutions without truncating, that is, by
\[
 \widetilde{\height}_\sha(\shf) \deq \min_{\shr} \max_{0\leq k,\, 1\leq l\leq m_k} d_{k,l}\ .
\]

On a general projective variety  there might be sheaves that do not possess finite locally free resolutions at all, and it can happen that the  height of a sheaf is infinite if we do
not truncate resolutions.  
\end{rmk}


A result of Arapura \cite[Corollary 3.2]{AraFrob} gives effective estimates on the height of a coherent sheaf in terms of its Castelnuovo--Mumford regularity.

\begin{lem}\label{lem:Arapura}
Let $X$ be an irreducible projective variety, $\sha$ an ample and globally generated line bundle, $\shf$ a coherent sheaf on $X$. Given a natural number $k$, there exist
vector spaces $V_i$ for $1\leq i\leq k$ and a resolution
\[
 V_k\otimes\sha^{\otimes -r_\shf-kr_X} \to\ldots\to V_1\otimes\sha^{\otimes -r_\shf-r_X} \to V_0\otimes\sha^{\otimes -r_\shf} \to\shf\to 0\ ,
\]
where 
\[
 r_\shf \deq \reg_\sha(\shf)\ \ \ \text{and}\ \ \ r_X \deq \max\st{1,\reg_\sha(\OO_X)}\ .
\]
\end{lem}
\begin{proof}
Without loss of generality we can assume that $\shf$ is $0$-regular by replacing $\shf$ by $\shf\otimes\sha^{r_\shf}$, thus we can assume $r_\shf=0$. Consequently,
$\shf$ is globally generated by Mumford's theorem; set $V_0 \deq \HH{0}{X}{\shf}$, and 
\[
 \shk_0 \deq \shf\otimes\sha^{\otimes - r_X}\ , \ \shk_1 \deq \ker \zj{ V_0\otimes\OO_X\twoheadrightarrow \shk_0\otimes\sha^{\otimes r_X}}\ .
\]
A quick cohomology computation \cite[Lemma 3.1]{Arapura} shows that $\shk_1$ is $r_X$-regular, hence we can repeat the above process for $\shk_1$ in place of 
$\shf$. This leads to a sequence of vector spaces $V_i$, and sheaves $\shk_i$ which fit into the exact sequences
\[
 0 \lra \shk_{i+1} \lra V_i\otimes \OO_X \lra \shk_i\otimes \sha^{\otimes r_X} \lra 0\ ,
\]
or, equivalently, 
\[
  0 \lra \shk_{i+1}\otimes \sha^{\otimes -ir_X} \lra V_i\otimes \sha^{\otimes -ir_X} \lra \shk_i\otimes \sha^{\otimes (1-i)r_X} \lra 0\ .
\]
We obtain the statement of the Lemma by combining these sequences into the required resolution. 
\end{proof}

\begin{cor}
 With notation as above, the height of an $r_\shf$-regular coherent sheaf $\shf$ is
\[
 \height_\sha (\shf) \dleq r_\shf + r_X\cdot\dim X\ .
\]
\end{cor}

\begin{prop}[Properties of height]\label{prop:height}
Let $X$ be an irreducible projective variety of dimension $n$, $\sha$ an ample line bundle. Then, the following hold.
\begin{enumerate}
 \item For coherent sheaves $\shf_1$ and $\shf_2$ we have 
\[
 \height_\sha(\shf_1\otimes\shf_2) \dleq \height_\sha(\shf_1) + \height_\sha(\shf_2)\ .
\]
\item There exists a positive constant $M=M(X,\sha)$ having the property that 
\[
 \height_\sha(\shn) \dleq M
\]
for all numerically trivial line bundles $\shn$ on $X$. 
\end{enumerate}
\end{prop}
\begin{proof}
The first statement is an immediate consequence of the fact that the tensor product of appropriate resolutions of $\shf_1$ and $\shf_2$ is a resolution of $\shf_1\otimes\shf_2$ of 
the the required type. 

The second claim is a consequence of the fact that numerically trivial divisors on a projective variety are parametrized by a quasi-projective variety. 
Indeed, it follows by  Lemma~\ref{lem:USC CM} and the Noetherian property of the Zariski topology that there exists a constant $M'$ satisfying 
\[
 \reg_\sha(\shn) \dleq M'
\]
for all numerically trivial line bundles $\shn$. By the Corollary of Lemma~\ref{lem:Arapura},
\[
 \height_\sha(\shn) \dleq M \deq M'+ r_X\cdot \dim X\ ,
\]
 as required.
\end{proof}

\begin{lem}[Upper-semicontinuity of Castelnuovo--Mumford regularity]\label{lem:USC CM}
Let $X$ an irreducible projective variety, $\sha$ an ample and globally generated line bundle on $X$. Given a flat family of line bundles	
$\shl$ on $X$ parametrized by a quasi-projective variety
$T$, the function
\[
 T \ni t \mapsto \reg_\sha(\shl_t) 
\]
is upper-semicontinuous.
\end{lem}
\begin{proof}
 Since Castelnuovo--Mumford regularity is checked by the vanishing of finite\-ly many line bundles, the statement follows from the semicontinuity theorem for
cohomology.
\end{proof}

\begin{eg}[Height and regularity on projective spaces]
Here we discuss the relationship between height and regularity with respect to $\OO(1)$ on an $n$-dimensional projective space $\PP$. We claim
that 
\[
 \height (\shf) \equ \reg (\shf) + n
\]
holds for an arbitrary coherent sheaf $\shf$ on $\PP$. 

For the inequality 
\[
 \height (\shf) \dleq \reg (\shf) + n
\]
observe that by  \cite[Proposition 1.8.8]{PAG} there must exist a long exact sequence
\[
 \ldots \lra  \bigoplus \OO_\PP(-\reg(\shf)-1) \lra \bigoplus \OO_\PP(-\reg(\shf)) \lra \shf \lra 0\ ,
\]
hence we are done by the definition of height. 

To see the reverse inequality, let 
\[
 \ldots \lra \shf_2 \lra \shf_1 \lra \shf_0 \lra \shf \lra 0
\]
be a resolution of $\shf$ with
\[
 \shf_k \equ \bigoplus_{i=1}^{k}\OO_\PP(-d_{k,i})\ .
\]
Since $d_{k,i} \leq \height (\shf)$ by definition, we obtain that $\shf_k$ is $\height (\shf) - k$ regular for all $n\geq k\geq 0$, therefore 
\[
 \reg (\shf) \dleq  \height (\shf) - n 
\]
according to  \cite[Example 1.8.7]{PAG}. 
\end{eg}

\begin{rmk}[Height and Serre vanishing with estimates]
The introduction of the height of a coherent sheaf leads to an effective version of Serre's vanishing theorem. 
With the notation of the introduction, one has 
\[
 m_0(\sha,\shf) \dleq m_0(\sha,\sha)+\height_\sha(\shf)\ .
\]
\end{rmk}

\begin{lem}[Demailly--Peternell--Schneider; \cite{DPS}, Proposition 1.2]\label{lem:DPS}
Let $\shl$ be a uniformly $q$-ample line bundle on $X$ with respect to an ample line bundle $\sha$ for a given constant $\la=\la(\sha,\shl)$. 
Given a coherent sheaf $\shf$ on $X$,
\[
 \HH{i}{X}{\shll{m}\otimes\shf} \equ \{0\} 
\]
for all $i>q$ and $m\geq \la\cdot (\height_\sha(\shf)+1)$.
\end{lem}

\begin{proof}
Let 
\[
 \ldots\to\bigoplus_{1\leq l\leq m_k}\sha^{\otimes - d_{k,l}} \to\ldots\to \bigoplus_{1\leq l\leq m_0}\sha^{\otimes -d_{0,l}} \to\shf\to 0
\]
be a resolution where the value of $\height_\sha(\shf)$ is attained, and write $\shf_k$ for the image sheaf of the $k$\textsuperscript{th} differential 
in the above sequence (note that $\shf_0=\shf$). Chopping up the resolution of $\shf$ into short exact sequences yields
\[
 0 \lra \shf_{k+1} \lra \bigoplus_{1\leq l\leq m_k}\sha^{\otimes - d_{k,l}} \lra \shf_{k} \lra 0
\]
for all $0\leq k\leq \dim X$. By the uniform $q$-ampleness assumption on $\shl$ we obtain that 
\[
 \HH{i}{X}{\shll{m}\otimes \sha^{\otimes - d_{k,l}}} \equ \{0\}\ \ \ \text{for all $i>q$ and $m\geq \la (d_{k,l}+1)$.}
\]
By induction on $k$ we arrive at
\[
 \HH{i}{X}{\shll{m}\otimes \shf} \simeq\ldots\simeq \HH{i+k}{X}{\shll{m}\otimes \shf_k} \simeq \HH{i+k+1}{X}{\shll{m}\otimes\shf_{k+1}}
\]
for all $i>q$ and $m\geq \la (\height_\sha(\shf)+1)$. The statement of the Lemma follows by taking $k=\dim X$. 
\end{proof}

The idea for the following modification of the proof of \cite[Proposition 1.4]{DPS} was suggested to us by Burt Totaro.

\begin{proof}[Proof of Theorem~\ref{thm:q-ample open}]
Fix an integral ample divisor $A$, as well as integral Cartier divisors $B_1,\dots,B_\rho$ whose numerical equivalence classes form a basis of the  rational
N\'eron--Severi space. 
Let $D$ be an integral uniformly $q$-ample divisor (for a constant $\la=\la(D,A)$),   $D'$ a  $\QQ$-Cartier divisor, and write 
\[
 D' \,\equiv\, D + \sum_{i=1}^{\rho}\la_iB_i 
\]
for rational numbers $\la_i$. Let $k$ be a positive integer clearing all denominators, then,
\[
 kD' \equ kD +\sum_{i=1}^{\rho}k\la_iB_i + N
\]
for a numerically trivial (integral) divisor $N$. We want to show that 
\[
 \HH{i}{X}{mkD'-pA} \equ \{0\}
\]
whenever $m\geq \la(D',A)\cdot p$ for a suitable positive constant $\la$. By Lemma~\ref{lem:DPS} applied with
\[
\shl\equ \OO_X(D)\ ,\ \sha\equ\OO_X(A)\ , \text{ and } \shf\equ \OO_X(\sum_{i=1}^{\rho}mk\la_iB_i + mN-pA)\ ,
\]
this will happen whenever 
\[
 m \geq \la(D,A)\cdot \height_A(\sum_{i=1}^{\rho}mk\la_iB_i + mN-pA)\ .
\]
Observe that 
\begin{eqnarray*}
 \height_A(\sum_{i=1}^{\rho}mk\la_iB_i + mN-pA) & = & \sum_{i=1}^{\rho}\height_A(mk\la_iB_i) + \height_A(mN) + \height_A(-pA)  \\
& \leq  & \sum_{i=1}^{\rho} mk|\la_i|\cdot\max\st{\height_A(B_i),\height_A(-B_i)} + M  + p \ ,
\end{eqnarray*}
where $M$ is the constant from Proposition~\ref{prop:height}; note that $\height_A(-pA)=p$ for $p\geq 0$. Therefore, if the $\la_i$'s are close enough to zero so that 
\[
 \la\cdot \sum_{i=1}^{\rho} k|\la_i|\cdot\max\st{\height_A(B_i),\height_A(-B_i)} \ < \ \frac{1}{2}\ ,
\]
then it suffices to require
\[
 m \geq 2\la (M+p)\ ,
\]
and $D'$ will be $q$-ample. This shows the upper-semicontinuity of uniform $q$-ample\-ness. 
\end{proof}

\begin{rmk}
If $D_1$ is a $q$-ample and $D'$ is an $r$-ample divisor, then their sum $D+D'$ can only be guaranteed to be $q+r$-ample; this bound is sharp as shown
in \cite[Sect.~8]{Totaro}.  As a consequence, the cone of $q$-ample $\RR$-divisor classes is not necessarily convex. We denote this cone by $\Amp^q(X)$.

Interestingly enough, if we restrict our attention to semi-ample divisors, then Sommese proves in \cite{Sommese} (see also Corollary~\ref{cor:sum semiample}
below) that the sum of $q$-ample divisors retains this property. 
\end{rmk}

It is an interesting question how to characterize the cone of $q$-ample divisors for a given integer $q$. If $q=0$, then the Cartan--Serre--Grothendieck theorem implies that the $\Amp^0(X)$ equals
the ample cone.

Totaro describes the $(n-1)$-ample cone with the help of duality theory.

\begin{thm}[Totaro; \cite{Totaro}, Theorem 10.1]\label{thm: n-1 ample}
For an irreducible projective variety $X$ we have
\[
 \Amp^{n-1}(X) \equ N^1(X)_\RR \setminus (-\overline{\Eff(X)})\ . 
\]
\end{thm}

\begin{cor}[$1$-ampleness on surfaces]
 If $X$ is a surface, then a divisor $D$ on $X$ is $1$-ample if and only if $(D\cdot A)>0$ for some ample divisor $A$ on $X$. 
\end{cor}

\begin{rmk}
The cone of $q$-ample divisors on a $\QQ$-factorial projective toric variety has been shown to be polyhedral 
(more precisely, to be the interior of the union of finitely many rational polyhedral cones)  by Broomhead and Prendergast-Smith \cite[Theorem 3.3]{BrPS}. 
Nevertheless, an explicit combinatorial description in terms of the fan of the underlying toric variety  along the lines  of \cite{HKP} is not yet known. 
\end{rmk}

Totaro links partial positivity to the vanishing of higher asymptotic cohomology. Generalizing the main result of \cite{dFKL} (see also \cite{ACF}
for terminology),  he asks the following question.

\begin{question}[Totaro]
Let $D$ be an $\RR$-divisor class on a complex projective variety, $0\leq q\leq n$ an integer. Assume that $\ha{i}{X}{D'}=0$ for all $i>q$ and all $D'\in N^1(X)_\RR$ in a neighbourhood of 
$D$. Is is true that $D$ is $q$-ample? 
\end{question}

\begin{rmk}
Broomhead and Prendergast-Smith \cite[Theorem 5.1]{BrPS} answered  Totaro's question positively for toric varieties. 
\end{rmk}

It is expected the $q$-ampleness should have more significance in the case of big line bundles. A first move in this direction comes from the following 
Fujita-type vanishing statement (see \cite{Fujita} or \cite[Theorem 1.4.35]{PAG} for Fujita's original statement).

\begin{thm}[\cite{qAmp}, Theorem C]\label{thm:Serre-Fujita for big}
Let $X$ be a complex projective scheme, $L$ a big Cartier divisor, $\shf$ a coherent sheaf on $X$. Then there exists a positive integer $m_0(L,\shf)$ such that 
\[
 \HH{i}{X}{\shf\otimes\OO_X(mL+D)} \equ \{0\}
\]
 for all $i>\dim \Bplus (L)$, $m\geq m_0(L,\shf)$, and all nef divisors $D$ on $X$.
\end{thm}

\begin{rmk}[Augmented base loci on schemes]
The augmented base locus of a $\QQ$-Cartier divisor  $L$ is defined in \cite{ELMNP}  via
\[
 \Bplus (L) \deq \bigcap_{A} \B(L-A)	
\]
where $A$ runs through all ample $\QQ$-Cartier divisors. As opposed to the stable base locus of a divisor, the augmented base locus is invariant 
with respect to numerical equivalence of divisors. The augmented base locus of a $\QQ$-divisor $L$  is empty precisely if $L$ is ample. 

Although it is customary to define the stable base locus and the augmented base locus of a divisor in the setting of projective varieties, as it is pointed out in
\cite[Section 3]{qAmp}, these notions make perfect sense on more general schemes. 

For an invertible sheaf $\shl$ on an arbitrary scheme $X$, let us denotes by  $\shf_\shl$ the quasi-coherent subsheaf of $\shl$ generated by $\HH{0}{X}{\shl}$. 
then we can set 
\[
\mathfrak{b}(\shl) \deq \ann_{\OO_X} (\shl/\shf_\shl)\ , 
\]
and define $\Bs(\shl)$ to be the closed subscheme corresponding to $\mathfrak{b}(\shl)$; furthermore we define  
\[
\B(\shl) \deq  \bigcap_{m=1}^{\infty} \Bs(\shl^{\otimes m})_{\text{red}} \dsubseteq X
\]
as a closed subset of the topological space associated to $X$. All basic properties of the stable base locus are retained (see again \cite[Section 3]{qAmp}), in particular,
if $X$ is complete and algebraic over $\CC$ (by which we mean  separated, and of finite type over $\CC$), then we recover the original definition of stable base loci. 

Assuming $X$ to be projective and algebraic over $\CC$, we define the augmented base locus of a $\QQ$-Cartier divisor  $L$ via
\[
 \Bplus (L) \deq \bigcap_{A} \B(L-A)	
\]
where $A$ runs through all ample $\QQ$-Cartier divisors. Again, basic properties are preserved, and in the case of projective varieties we recover the original
definition.
\end{rmk}

An interesting further step in this direction is provided 
by Brown's work, where he connects $q$-ampleness of a big line bundle to its behaviour when restricted to  its augmented base locus. 

\begin{thm}[Brown; \cite{Brown}, Theorem 1.1]
Let $\shl$ be a big line bundle on a complex projective scheme $X$, denote by $\Bplus (\shl)$ the augmented base locus of $\shl$. For a given integer $0\leq q\leq n$, $\shl$ is $q$-ample
if and only if $\shl|_{\Bplus(\shl)}$ is $q$-ample.    
\end{thm}

We give a very rough outline of the proof of Brown's result. First, if $\shl$ is $q$-ample on $X$, and $Y\subseteq X$ denotes $\Bplus(\shl)$ with the reduced 
induced scheme structure, then the projection formula and the preservation of cohomology groups under push-forward by closed immersions imply that 
$\shl|_{\Bplus(\shl)}$
is $q$-ample as well. 

The other implication comes from the following  useful observation from \cite{Brown}, a  restriction theorem for line bundles that are not $q$-ample 
\cite[Theorem 2.1]{Brown}: let $\shl$ be a line bundle on a reduced projective scheme $X$, which is not $q$-ample, and let $\shl'$ be a line bundle on $X$ with a 
nonzero section $s$ having the property that  $\shll{a}\otimes\shl'^{\otimes -b}$ is ample for some positive integers $a,b$. Then, $\shl|_{Z(s)}$ is not $q$-ample.

\subsection{Sommese's geometric $q$-ampleness}

In this section, we will discuss Sommese's geometric notion of  $q$-ampleness, and relate it  to the more cohomologically oriented discussion in the previous sections. 

\begin{defi}[Sommese; \cite{Sommese}, Definition~1.3]\label{defi:Sommese geom}
 Let $X$ be a projective variety, $\shl$ a line bundle on $X$. We say that $\shl$ is \emph{geometrically $q$-ample} for a natural number $q$, if 
 \begin{enumerate}
  \item $\shl$ is semi-ample, i.e., $\shll{m}$ is globally generated for some natural number $m \geq 1$, 
  \item the maximal fibre dimension of $\phi_{|\shll{m}|}$ is at most $q$.
 \end{enumerate}
\end{defi}

More generally, Sommese defines a vector bundle $\she$ over $X$ to be  geometrically $q$-ample, if $\OO_{\PP(\she)}(1)$ is geometrically
$q$-ample, and goes on to prove  many interesting results for vector bundles (see \cite[Proposition 1.7]{Sommese} or \cite[Proposition 1.12]{Sommese}
for instance). In this paper we will only treat the line bundle case.

\begin{rmk}[Iitaka fibration]
We briefly recall the semi-ample or Iitaka fibration associated to a semi-ample line bundle $\shl$ on a normal projective variety $X$ \cite[Theorem 2.1.27]{PAG}: 
there exists an algebraic fibre space (a surjective projective morphism with connected fibres) $\phi\colon X\to Y$ with the property that for any sufficiently
large and divisible $m\in \NN$ one has 
\[
 \phi_{|\shll{m}|} \equ \phi \ \ \ \text{and}\ \ \ Y_m\equ Y\ ,
\]
where $Y_m$ denotes the image of $X$ under $ \phi_{|\shll{m}|}$. 

In addition there exists an ample line bundle $\sha$ on $X$ such that 
\[
 \phi^*\sha \equ \shll{k}
\]
for a suitable positive integer $k$. 
\end{rmk}

\begin{rmk}
If $X$ is a normal variety then Sommese's conditions are equivalent to requiring that $\shl$ is semi-ample and its semi-ample fibration has fibre dimension at most 
$q$. In the case when $X$ is not normal, it is  a priori not clear if the set of integers $q$ for which $\shl$ is $q$-ample depends on the choice of $m$; a posteriori
this follows from Sommese's theorem. Nevertheless, for this reason the definition via the Iitaka fibration is cleaner in the case of normal varieties. 
\end{rmk}

 We summarize Sommese's results in this direction.

\begin{thm}[Sommese; \cite{Sommese}, Proposition 1.7]\label{thm:Sommese equivalence}
Let $X$ be a projective variety, $\shl$ a semi-ample line bundle over $X$ with Iitaka fibration $\phi_\shl$. For a natural number $q$ 
the following are equivalent.
\begin{enumerate}
 \item[(i)] The line bundle $\shl$ is geometrically $q$-ample.
 \item[(ii)] The maximal dimension of an irreducible subvariety $Z\subseteq X$ with the property that  $\shl|_Z$ is trivial is at most $q$.
 \item[(iii)] If $\psi:Z\to X$ is a morphism from a projective variety $Z$ such that $\phi^*\shl$ is trivial, then $\dim Z\leq q$.
 \item[(iv)] The line bundle $\shl$ is naively $q$-ample, that is, for every coherent sheaf $\shf$ on $X$ there exists a natural number $m_0=m_0(\shl,\shf)$ 
 with the property that 
 \[
  \HH{i}{X}{\shf\otimes \shl^{\otimes m}} \equ 0 \ \ \ \text{for all $i>q$ and $m\geq m_0$.}
 \]
\end{enumerate}
\end{thm}
\begin{proof}
All equivalences are treated in \cite{Sommese}, here we will describe the equivalence between (i) and (iv), that is, we prove that for semi-ample line bundles geometric
and cohomological $q$-ampleness agree. 

Let $m_0\geq 1$ be an integer for which $\shll{m_0}$ is globally generated, and let $\phi\colon X\to Y\subseteq \PP$ denote the associated morphism. Then there exists 
an ample line bundle $\sha$ on $Y$ and a positive integer $k$ such that $\shll{k}=\phi^*\sha$. Fix a coherent sheaf $\shf$ on $X$, and consider the Leray spectral
sequence
\[
 \HH{p}{Y}{R^r\phi_*(\shf\otimes \shll{m})} \Longrightarrow \HH{p+r}{X}{\shf\otimes\shll{m}}\ .
\]
Let us write $m=sk+t$ with $0\leq t<k$ and $s\geq 0$ integers. For the cohomology groups on the left-hand side
\begin{eqnarray*}
 \HH{p}{Y}{R^r\phi_*(\shf\otimes \shll{m})} & = &  \HH{p}{Y}{R^r\phi_*(\shf\otimes (\shll{k})^{\otimes s}\otimes\shll{t})} \\
 & = &  \HH{p}{Y}{R^r\phi_*(\shf\otimes \shll{t})\otimes \sha^{\otimes s}} 
\end{eqnarray*}
by the projection formula. Serre's vanishing theorem yields 
\[
 \HH{p}{Y}{R^r\phi_*(\shf\otimes \shll{t})\otimes \sha^{\otimes s}}  \equ 0 \ \ \text{for all $p\geq 1$, $s\gg 0$ and all $0\leq t<k$},
\]
hence 
\[
 \HH{0}{Y}{R^r\phi_*(\shf\otimes \shll{m})} \simeq \HH{r}{X}{\shf\otimes\shll{m}}
\]
for $m\gg 0$. On the other hand, if the maximal fibre dimension of $\phi$ is $q$, then $R^r\phi_*(\shf\otimes \shll{m})=0$ for all $r>q$, therefore 
\[
 \HH{r}{X}{\shf\otimes\shll{m}} \equ 0 \ \ \text{for all $m\gg 0$ and $r>q$}
\]
and consequently $\shl$ is naively $q$-ample as claimed. 

For the other implication assume that $\shl$ is not geometrically $q$-ample, hence has a fibre $F\subseteq X$ of dimension $f>q$. Starting from here one constructs
a coherent sheaf $\shf$ on $F$ having the property that $R^f\phi_*\shf$ is a skyscraper sheaf, which, by the Leray spectral sequence above, would imply that 
\[
 \HH{f}{X}{\shf\otimes\shll{ks}} \,\neq\, 0\ \ \text{for all $s\geq 1$.}
\]
Without loss of generality we can assume that $F$ is irreducible, otherwise we replace it by one of its top-dimensional irreducible components. 
Let $\pi\colon \widetilde{F}\to F$ denote a resolution of singularities of $F$, and consider the Grauert--Riemenschneider canonical sheaf 
\[
 \shf \deq \shk_{\widetilde{F}/F} \deq \pi_*\omega_{\widetilde{F}}\ .
\]
By Grauert--Riemenschneider \cite{GrRiem}
\[
 R^f(\phi|_F)_*\shf \simeq \HH{f}{F}{\shf} \,\neq\,  0 
\]
as $\phi|_F$ maps $F$ to a point. 
\end{proof}

\begin{cor}\label{cor:sum semiample}
 Let $D_1$ and $D_2$ be geometrically $q$-ample divisors on a smooth variety. Then so is $D_1+D_2$.
\end{cor}
\begin{proof}
 This is \cite[Corollary 1.10.2]{Sommese}.
\end{proof}

Additionally, Kodaira-Akizuki-Nakano vanishing continues to hold for the expected range of cohomology groups and degrees of differential forms:
\begin{thm}[Kodaira--Akizuki--Nakano for geometrically $q$-ample bundles]\label{thm:qKodairageometric}
Let $\shl$ be a geometrically $q$-ample line bundle on a smooth projective variety $X$. Then,
\[
 \HH{i}{X}{\wedge^j\Omega_X\otimes\shl} \equ \{0\} \ \ \ \text{for all $i + j>\dim X + q$.}
\]
In particular, the following \emph{$q$-Kodaira vanishing} holds:
\[\HH{i}{X}{\omega_X \otimes\shl} \equ \{0\} \ \ \ \text{for all $i > q$.}\]
\end{thm}
\begin{proof}
This is proven in \cite[Proposition 1.12]{Sommese}.
\end{proof}
\begin{rmk}
 Sommese's version of the Kodaira-Akizuki-Nakano vanishing was later shown by Esnault and Viehweg to hold for an even larger range of values for $i$ and $j$, see \cite{EsnaultViehwegOriginal} and \cite[Corollary~6.6]{EsnaultViehweg}.
\end{rmk}

In Example~\ref{ex:counterexampleOttem} and Section~\ref{sect:qKodaira} below we discuss the question whether Kodaira vanishing still continues to hold when one drops the semiampleness condition, i.e., for general $q$-ample line bundles.

\subsection{Ample subschemes, and a Lefschetz hyperplane theorem for $q$-ample divisors}\label{subsect:Lefschetz}

Building upon the theory of $q$-ample line bundles and  Hartshorne's classical work \cite{HS_Ample}, Ottem \cite{Ottem} defines the notion of an ample 
subvariety (or subscheme),  and goes on to verify that ample subvarieties share many of the significant algebro-geometric and topological properties of their codimension-one counterparts.
One of the highlights of his work is a Lefschetz-type hyperplane theorem, which we will use for a proof of Kodaira's vanishing theorem for 
effective $q$-ample Du Bois divisors in Section~\ref{sect:qKodaira} below. Here we briefly recall the theory obtained in \cite{Ottem}. 

\begin{defi}[Ample subschemes]\label{defi:amplesubscheme}
Let $X$ be a projective scheme, $Y\subseteq X$ a closed subscheme of codimension $r$, $\pi:\widetilde{X}\to X$ the blow-up of $Y$ with exceptional divisor 
$E$. Then, $Y$ is called \emph{ample} if $E$ is $(r-1)$-ample on $\widetilde{X}$. 
\end{defi}

The idea behind this notion is classical: it has been known for a long time that positivity properties of $Y$ can be often read off from the geometry of 
the complement $X-Y\simeq \widetilde{X}-E$. In spite of this, the concept has not been defined until recently. 
\begin{eg}
As it can be expected, linear subspaces of projective spaces are ample.
\end{eg}

\begin{rmk}[Cohomological dimension of the complement of an ample subscheme]\label{rmk:cd of complement}
An important geometric feature of ample divisors is that their complement is affine. In terms of cohomology, this is equivalent to requiring
\[
 \HH{i}{X}{\shf} \equ \{0\} \ \ \ \text{for all $i>0$, $\shf$ coherent sheaf on $X$.}
\]
 If we denote as customary the cohomological dimension of a subset $Y\subseteq X$ by $\cd(Y)$, then we can phrase Ottem's generalization \cite[Proposition~5.1]{Ottem} to the $q$-ample 
case as follows: if $U\subseteq X$ is an open subset of a projective scheme $X$ having the property that $X\setminus U$ is the support of a $q$-ample divisor, 
then 
\[
 \cd(U) \dleq q\ .
\]
\end{rmk}

The observation on cohomological dimensions of complements leads to the following Lefschetz-type statement. 

\begin{thm}[Generalized Lefschetz hyperplane theorem, Ottem; \cite{Ottem}, Corollary 5.2]\label{thm:gen Lefschetz}
Let $D$ be an effective $q$-ample divisor on a projective variety $X$ with smooth complement. Then, the restriction morphism 
\[
 \HH{i}{X}{\QQ} \lra \HH{i}{D}{\QQ} \text{\ is\ }\ \  \begin{cases} \text{an isomorphism for $0\leq i\leq n-q-1$,} \\ \text{injective for $i=n-q-1$}.
                                    \end{cases}
\]
\end{thm}

We give Ottem's proof to show the principles at work. 
\begin{proof}
Via the long exact sequence for relative cohomology and Lefschetz duality, the statement reduces to the claim that 
\[
 \HH{i}{X\setminus D}{\CC} \equ \{0\} \ \ \ \text{for $i>n+q$}.
\]
 This latter follows from the Fr\"olicher spectral sequence
\[
 E_1^{st} \equ \HH{s}{X \setminus D}{\Omega^t_{X \setminus D}} \Longrightarrow \HH{s+t}{X\setminus D}{\CC},
\]
as 
\[
 \HH{s}{X\setminus D}{\Omega^t_{X \setminus D}} \equ \{0\} \ \  \ \text{for all $s+t>n+q$}
\]
by Remark~\ref{rmk:cd of complement}.
\end{proof}

\begin{cor}[Lefschetz hyperplane theorem for ample subvarieties; \cite{Ottem}, Corollary 5.3]
Let $Y$ be an ample  local complete intersection subscheme in a smooth complex projective variety $X$. Then, the restriction morphism
\[
 \HH{i}{X}{\QQ}\lra \HH{i}{Y}{\QQ}\ \ \text{is\ \ } \begin{cases} \text{an isomorphism} & \ \text{if $i<\dim Y$} \\
                                                 \text{injective} & \ \text{if $i=\dim Y$.}
                                                \end{cases}
\]
\end{cor}

The following is a summary of properties of smooth ample subschemes.

\begin{thm}[Properties of smooth ample subschemes, \cite{Ottem}, Corollary 5.6 and Theorem 6.6]
Let $X$ be a smooth projective variety, $Y\subseteq X$ a non-singular ample subscheme of dimension $d\geq 1$. Then, the following hold:
\begin{enumerate}
 \item The normal bundle $N_{Y/X}$ of $Y$ is an ample vector bundle.
 \item For every irreducible $(\dim X-d)$-dimensional subvariety $Z\subseteq X$ one has $(Y\cdot Z)>0$. In particular, $Y$ meets every divisor.
 \item The Lefschetz hyperplane theorem holds for rational cohomology holds on $Y$:
 \[
 \HH{i}{X}{\QQ}\lra \HH{i}{Y}{\QQ}\ \ \text{is\ \ } \begin{cases} \text{an isomorphism} & \ \text{if $i<\dim Y$} \\
                                                 \text{injective} & \ \text{if $i=\dim Y$.}
                                                \end{cases}
\]
\item Let $\widehat{X}$ denote the completion of $X$ with respect to $Y$. For any coherent sheaf $\shf$ on $X$ one has 
\[
 \HH{i}{X}{\shf}\lra \HH{i}{\widehat{X}}{\shf}\ \ \text{is\ \ } \begin{cases} \text{an isomorphism} & \ \text{if $i<\dim Y$} \\
                                                 \text{injective} & \ \text{if $i=\dim Y$.}
                                                \end{cases}
\]
\item The inclusion $Y\hookrightarrow X$ induces a surjection
\[
 \pi_1(Y)\twoheadrightarrow \pi_1(X)
\]
on the level of fundamental groups. 
\end{enumerate}
\end{thm}

\begin{proof}
 We will only discuss (5), since it is the only statement that is slightly different from its original source. Because $Y$ is smooth over a reduced
base, it is  automatically reduced, and by smoothness again, it is irreducible exactly when it is connected. But (3) in the case of $i=0$
implies that $Y$ is connected. The rest follows from Ottem's proof.  
\end{proof}

Based on the fact that a Lefschetz-type theorem holds for $q$-ample divisors, as seen in Theorem~\ref{thm:gen Lefschetz}, and that this forms one of the ingredients of the proof of the Kodaira vanishing theorem in the classical setup, cf.~\cite[Section~4.2]{PAG}, as well as on the fact that $q$-Kodaira vanishing continues to hold for \emph{geometrically} $q$-ample divisors, Theorem~\ref{thm:qKodairageometric}, one might hope that there is a $q$-Kodaira vanishing theorem for $q$-ample divisors. The following example shows that this is not true in general.
\begin{eg}[Counterexample to Kodaira vanishing for non-pseudo-effective $q$-ample divisors, Ottem; \cite{Ottem}, Section 9] \label{ex:counterexampleOttem}
Let $G=SL_3(\CC)$, $B\leqslant G$ the Borel subgroups consisting of upper triangular matrices, and consider the homogeneous space
$G/B$. By the Bott--Borel--Weil theorem and a brief computation Ottem shows the existence of a non-pseudo-effective line bundle $\shl$ on $G/B$, which is $1$-ample, but for which the cohomology group $\HH{2}{G/B}{\omega_{G/B} \otimes \shl}$ does not vanish.
\end{eg}

It turns out, however, that by putting geometric restrictions on the $q$-ample divisor in question, one can in fact prove Kodaira-style vanishing theorems.
We will do this in the next section.

\section{$q$-Kodaira vanishing for Du Bois divisors and log canonical pairs}\label{sect:qKodaira}

This section contains the proofs of various versions of Kodaira's vanishing theorem for $q$-ample divisors. First we present the argument in the smooth case, where the reasoning is particularly transparent and simple.

\begin{thm}\label{thm:Kodaira smooth}
Let $X$ be a smooth projective variety, $D$ a smooth reduced effective $q$-ample divisor on $X$. Then,
\[
 \HH{i}{X}{\OO_X(K_X+D)} \equ \{0\}\ \ \ \text{for all $i>q$}.
\]
\end{thm}
\begin{proof}
By Serre duality, it suffices to show
\[
 \HH{i}{X}{\OO_X(-D)} \equ \{0\} \ \ \ \text{for all $i\leq n-q-1$}\ . 
\]
To this end,  follow the proof of Kodaira vanishing in \cite[Section 4.2]{PAG} with minor modifications. Ottem's 
generalized Lefschetz hyperplane theorem for effective $q$-ample divisors, Theorem~\ref{thm:gen Lefschetz}, asserts that 
\[
\HH{i}{X}{\CC} \lra \HH{i}{D}{\CC} 
\]
is an isomorphism for $0\leq i  < n-q-1$ and an injection for $i=n-q-1$.

The Hodge decomposition then gives rise to homomorphisms
\begin{equation}\label{eq:HodgeDecomposition}
 r_{k,l}: \HH{l}{X}{\Omega_X^k} \lra \HH{l}{D}{\Omega_D^k}
\end{equation}
for which $r_{k,l}$ is an isomorphism for $0\leq k+l<n-q-1$, and  an injection for $k+l=n-q-1$.

Consequently, one has that $r_{o,l}:\HH{l}{X}{\OO_X}\to \HH{l}{D}{\OO_D}$ is an isomorphism for $l<n-q-1$, and is injective for 
$l=n-q-1$. 
 
Finally, consider the exact sequence 
\[
 \ses{\OO_X(-D)}{\OO_X}{\OO_D}\ .
\]
The properties of $r_{0,l}$ applied to the associated long exact sequence imply that 
\[
 \HH{i}{X}{\OO_X(-D)} \equ \{0\}\ \ \ \text{for all $i\leq n-q-1$}
\]
as we wished. 
\end{proof}
\begin{rmk}
 Our proof here should also be compared with the discussion in \cite[\S 4]{EsnaultViehweg}, where bounds on the cohomological dimension of the complement of a smooth divisor are used in a similar way to derive vanishing theorems of the type considered here.
\end{rmk}

\begin{rmk}
 We note that we have used only part of the information provided by the homomorphisms \eqref{eq:HodgeDecomposition}. 
 The remaining instances lead to a Akizuki-Nakano-type vanishing result for effective reduced smooth $q$-ample divisors; cf.~the discussion in \cite[Section~4.2]{PAG}. 
 Note however that a generalisation to the singular setup cannot be expected, as already for ample divisors on Kawamata log terminal varieties the natural generalisation of Kodaira--Akizuki--Nakano 
 vanishing does not hold in general; we refer the reader to \cite[Section~4]{GKP11} for a discussion and for explicit counterexamples.
\end{rmk}

One of the original contributions of our work is the observation that the argument in the proof of Theorem~\ref{thm:Kodaira smooth} can be modified to go through in the Du Bois case, in particular, for 
log canonical pairs on smooth projective varieties. Our discussion here is very much influenced by ``Koll\'ar's principle'', cf.~\cite[Section~3.12]{ReidChapters}, 
that vanishing occurs when a coherent cohomology group has a topological interpretation, see \cite[Section~5]{KollarHigherVanishingII}. For this, we heavily depend on the discussion of vanishing theorems in \cite{KollarShafarevich}.

\begin{thm}[$q$-Kodaira vanishing for reduced effective Du Bois divisors]\label{thm:reduceddivisor}
 Let $X$ be a normal proper variety, $D\subset X$ a reduced effective (Cartier) divisor such that
\begin{enumerate}
 \item $\mathcal{O}_X(D)$ is $q$-ample,
 \item $X \setminus D$ is smooth,
 \item $D$ (with its reduced subscheme structure) is Du Bois.
\end{enumerate}
Then, we have
\begin{align}
 &\HH{j}{X}{\omega_X \otimes \mathcal{O}_X(D)}  = \{0\} \quad \text{ for all } j> q, \label{eq:DuBoisEq1}
\intertext{as well as}
&\HH{j}{X}{\mathcal{O}_X(-D)} = \{0\} \quad \text{ for all }j < \dim X - q .\label{eq:DuBoisEq2}
\end{align}
\end{thm}
We refer the reader to \cite{KovacsSchwede} and \cite[Chapter 12]{KollarShafarevich} for introductions to the theory of Du Bois singularities, as well as to \cite{KovacsIntuitiveDuBois} for a simple characterisation of the Du Bois property for projective varieties, related to the properties of Du Bois singularities used here.
\begin{proof}
 Since $D$ is effective and $q$-ample, and $X \setminus D$ is smooth, Theorem~\ref{thm:gen Lefschetz} states that the restriction morphism
\begin{equation*}
\begin{xymatrix}{
 \HH{j}{X}{\mathbb{C}} \ar^{\Phi_j}[r]&  \HH{j}{D}{\mathbb{C}}
}\end{xymatrix}
\end{equation*}
is an isomorphism for $0\leq j < n-q-1$, and injective for $j=n-q-1$. If $F^\bullet$ denotes Deligne's Hodge filtration on $\HH{j}{D}{\mathbb{C}}$ 
with associated graded pieces $Gr^\bullet_F$, then without any assumption on $D$,  the natural map 
$\alpha_j:\HH{j}{D}{\mathbb{C}} \to Gr^0_F\bigl( \HH{j}{D}{\mathbb{C}}\bigr)$ factors as
\begin{equation}
\begin{gathered}
\begin{xymatrix}{
  &   & \\
 \HH{j}{D}{\mathbb{C}} \ar@/^2pc/[rr]^{\alpha_j} \ar[r]^>>>>>{\beta_j} & \HH{j}{D}{\mathcal{O}_D}\ar[r]^>>>>>{\gamma_j} & Gr^0_F\bigl(  \HH{j}{D}{\mathbb{C}}\bigr)
}
\end{xymatrix}
\end{gathered}
\end{equation}
for each $j$.

Moreover, since $D$ is assumed to be Du Bois, the map $\gamma_j$ is an isomorphism, see \cite[Section~1]{KovacsIntuitiveDuBois}. By abuse of notation, the Hodge filtration on $\HH{j}{X}{\mathbb{C}}$ 
will likewise be denoted by $F^\bullet$. By standard results of Hodge theory, e.g.~see \cite[Theorem~5.33.iii)]{PetersSteenbrink}, the map $\Phi_j$ is morphism of Hodge structures; in particular it is compatible with the filtrations $F^\bullet$ on $\HH{j}{X}{\mathbb{C}}$ and $\HH{j}{D}{\mathbb{C}}$. Hence, for each $j$, we obtain a natural commutative diagram 
\[
\begin{xymatrix}{
 \HH{j}{X}{\mathbb{C}}  \ar[r]^{\Phi_j} \ar@{->>}[d] &   \HH{j}{D}{\mathbb{C}} \ar@{->>}[d]^{\beta_j}\\
 \HH {j}{X}{\mathcal{O}_X}  \ar[r]^{\phi_j}  \ar[d]^{\cong} & \HH{j}{D}{\mathcal{O}_D} \ar[d]_\cong^{\gamma_j}\\
  Gr^0_F\bigl(\HH{j}{X}{\mathbb{C}})\bigr)\ar[r]& Gr_F^0 \bigl(\HH{j}{D}{\mathbb{C}} \bigr).
}
\end{xymatrix}
\]
Since $\Phi_j$ is an isomorphism for $0 \leq j < n-q-1$, and injective for $j= n-q-1$, Hodge theory, \cite[Corollaries~3.6 and~3.7]{PetersSteenbrink}, implies that the same is true for $\phi_j$; i.e., $\phi_j$ is an isomorphism for $0 \leq j < n-q-1$, and injective for $j = n-q-1$. 
Looking at the long exact cohomology sequence associated with 
\[0 \to \mathcal{O}_X(-D) \to \mathcal{O}_X \to \mathcal{O}_D \to 0\]
we conclude that 
\begin{equation}\label{eq:dualform}\HH{j}{X}{\mathcal{O}_X(-D)} = \{0\} \quad \text{ for }0\leq j \leq n-q-1,
\end{equation}
as claimed in equation \eqref{eq:DuBoisEq2}.

Now, as $X \setminus D$ is smooth and $D$ is Du Bois, $X$ itself has rational singularities by a result of Schwede \cite[Theorem~5.1]{Schwede}, see also \cite[Section~12]{KovacsSchwede}. In particular, $X$ is Cohen-Macaulay, \cite[Theorem~5.10]{KM98}, and hence we may apply Serre duality \cite[Theorem~5.71]{KM98} to equation \eqref{eq:dualform} to obtain
\[\HH{j}{X}{\omega_X \otimes \mathcal{O}_X(D)} = \{0\} \text { for all }j>q, \]
as claimed in equation \eqref{eq:DuBoisEq1}.
\end{proof}
\begin{cor}[$q$-Kodaira vanishing for reduced log canonical pairs]
  Let $X$ be a smooth projective variety, $D\subset X$ a reduced effective (Cartier) divisor such that
\begin{enumerate}
 \item $\mathcal{O}_X(D)$ is $q$-ample,
 \item the pair $(X, D)$ is log canonical.
\end{enumerate}
Then, we have
\begin{align}
 &\HH{j}{X}{\omega_X \otimes \mathcal{O}_X(D)}  = \{0\} \quad \text{ for all } j> q, 
\intertext{as well as}
&\HH{j}{X}{\mathcal{O}_X(-D)} = \{0\} \quad \text{ for all }j < \dim X - q .
\end{align}
\end{cor}
\begin{proof}
 Since $D$ is a union of log canonical centers of the pair $(X, D)$, it is Du Bois by \cite[Theorem~1.4]{KollarKovacsJAMS}. Hence, the claim follows immediately from Theorem~\ref{thm:reduceddivisor}.
\end{proof}
\begin{eg}
To give an example of a $q$-ample divisor that satisfies the assumptions of Theorems~\ref{thm:reduceddivisor} and \ref{thm:Kodaira smooth} above, and for which the desired vanishing does not follow directly from Sommese's results, let $Z \subset X$ be a smooth ample subscheme of pure codimension $r$ in a projective manifold $X$, and $\mathcal{O}_{\hat X}(E)$ the corresponding $(r-1)$-ample line bundle on the blow-up $\hat X$ of $X$ along $Z$, cf.~Example~\ref{ex:blowup}. Then, $\mathcal{O}_{\hat X}(E)$ is clearly not semiample, and hence not geometrically $q$-ample in the sense of Definition~\ref{defi:Sommese geom}. However, $E$ is effective, reduced, and smooth, and hence fulfills the assumptions of Theorems~\ref{thm:reduceddivisor} and \ref{thm:Kodaira smooth}. In this special case, the desired vanishing can also be derived from the results presented in \cite[\S 4]{EsnaultViehweg}.
\end{eg}

Finally, we will prove a version of the above for line bundles that are only $\mathbb{Q}$-effective. We start with the following slight generalisation of \cite[Theorem~12.10]{KollarShafarevich}.
\begin{prop}\label{prop:surjection}
 Let $X$ be a normal and proper variety, $\shl$ a rank one reflexive sheaf on $X$, and $D_i$ different irreducible Weil divisors on $X$. Assume that $\shl^{[m]}:= \bigl(\shl^{\otimes m}\bigr)^{**} \cong \mathcal{O}_X(\sum d_jD_j)$ for some integers $1 \leq d_j < m$. Set $m_j := m/\mathrm{gcd}(m, d_j)$. Assume furthermore that the pair $\left(X, \textstyle\sum (1- \frac{1}{m_j})D_j \right)$ is log canonical.
Then, for every $i\geq 0$ and $n_j \geq 0$, the natural map
\[\HH{i}{X}{\shl^{[-1]}(- \textstyle\sum n_j D_j)} \to \HH{i}{X}{\shl^{[-1]}} \]
is surjective.
\end{prop}
\begin{proof} Let $p: Y \to X$ be the normalisation of the cyclic cover corresponding to the isomorphism $\shl^{[m]} \cong \mathcal{O}_X(\sum d_jD_j)$. By \cite[Proposition~20.2]{FlipsAndAbundance}, 
we have 
\[K_Y = p^*\bigl(K_X + \textstyle\sum (1- \frac{1}{m_j})D_j  \bigr),\]
and therefore, $Y$ is log canonical by \cite[Proposition~20.3]{FlipsAndAbundance}. Hence, $Y$ is Du Bois by \cite[Theorem~1.4]{KollarKovacsJAMS}, and the natural map $\HH{j}{Y}{\mathbb{C}} \to \HH{j}{Y}{\mathcal{O}_Y}$ is surjective, cf.~the discussion in the proof of Theorem~\ref{thm:reduceddivisor}. Consequently, the assumptions of \cite[Theorem~9.12]{KollarShafarevich} are fulfilled. This implies the claim.
\end{proof}

We are now in the position to prove a version of $q$-Kodaira vanishing that works for $\mathbb{Q}$-effective line bundles. 

\begin{thm}[$q$-Kodaira vanishing for effective log canonical $\mathbb{Q}$-divisors]\label{thm:QQversion}
  Let $X$ be a proper, normal, Cohen-Macaulay variety, $\shl$ a $q$-ample line bundle on $X$, and $D_i$ different irreducible Weil divisors on $X$. Assume that 
$\shl^{m} \cong \mathcal{O}_X(\sum d_jD_j)$ for some integers $1 \leq d_j < m$. Set $m_j := m/\mathrm{gcd}(m, d_j)$. Assume furthermore that the pair $\bigl(X, \textstyle\sum (1- \frac{1}{m_j})D_j \bigr)$  is log canonical.
Then, we have
\begin{align}
 &\HH{j}{X}{\omega_X \otimes \mathcal{O}_X(D)}  = \{0\} \quad \text{ for all } j> q, \label{eq:lceq1}
\intertext{as well as}
&\HH{j}{X}{\mathcal{O}_X(-D)} = \{0\} \quad \text{ for all }j < \dim X - q .\label{eq:lceq2}
\end{align}
\end{thm}
\begin{proof} With Proposition~\ref{prop:surjection} at hand, the proof is the same as in the klt case, cf.~\cite[Chap.~10]{KollarShafarevich}. 
By Proposition~\ref{prop:surjection}, for all $i\geq 0$ and for all $k\in \mathbb{N}^{>0}$ such that that $k-1$ is divisible by $m$ we obtain a surjection
\[\HH{i}{X}{\shl^{-k}}\twoheadrightarrow \HH{i}{X}{\shl^{-1}}. \]
Since $X$ is Cohen-Macaulay, by Serre duality, \cite[Theorem~5.71]{KM98}, this surjection is dual to an injection
\begin{equation}\label{eq:injection}
\HH{n-i}{X}{\omega_X \otimes \shl} \hookrightarrow \HH{n-i}{X}{\omega_X \otimes \shl^k} \text{  for all }k \text{ as above}.
\end{equation}
As $\shl$ is $q$-ample, there exists a $k \gg 0$ such that 
\[\HH{n-i}{X}{\omega_X \otimes \shl^k}= \{0\} \text{ for all }n-i> q.\]
Hence, owing to the injection~\eqref{eq:injection}, we obtain 
\[\HH{j}{X}{\omega_X\otimes \shl} = \{0\} \text{ for all } j> q,\]
as claimed in equation \eqref{eq:lceq1}. The dual vanishing \eqref{eq:lceq2} then follows from a further application of Serre duality.
\end{proof}

\begin{rmk}For related work discussing ample divisors on (semi) log canonical varieties, the reader is referred to \cite{KSS10}.\end{rmk}

\begin{rmk}[Necessity of assumptions on the singularities]
 To see that some assumption on the singularities of the pair $(X, D)$ is necessary in Theorems~\ref{thm:reduceddivisor} and \ref{thm:QQversion}, we note that already Kodaira vanishing may fail already for ample line bundles on Gorenstein varieties with worse than log canonical singularities, see \cite[Example 2.2.10]{SommeseAdjunction}. Moreover, we note that for the dual form \eqref{eq:lceq2} of Kodaira vanishing the Cohen-Macaulay condition is strictly necessary: If $X$ is a projective variety with ample (Cartier) divisor $D$ for which \eqref{eq:lceq2} holds (with $q=0$), then $X$ is Cohen-Macaulay by \cite[Corollary~5.72]{KM98}.
\end{rmk}
\vspace{0.4cm}

\newcommand{\etalchar}[1]{$^{#1}$}
\providecommand{\bysame}{\leavevmode\hbox to3em{\hrulefill}\thinspace}
\providecommand{\MR}{\relax\ifhmode\unskip\space\fi MR }
\providecommand{\MRhref}[2]{%
  \href{http://www.ams.org/mathscinet-getitem?mr=#1}{#2}
}
\providecommand{\href}[2]{#2}

\begin{center}
 ------------------------
\end{center}

\end{document}